%% file: 2-ended.tex
\documentclass{article}
\usepackage{indentfirst,overpic}
\usepackage{amsthm,afterpage,float}
\usepackage{amsfonts}
\usepackage{amssymb}
\usepackage{overpic}

\newcommand{\COLORON}{1}
\newcommand{\NOTESON}{0}
\newcommand{\Debug}{0}

\input{defs}

\newcommand{\Gha}{\Gamma}
\newcommand{\bZ}{{\mathbb Z}}
\newtheorem{cla}{Claim}
\newcommand{\Aut}{\mbox{Aut}}
\newcommand{\Gpa}{G}
\newcommand{\sC}{\mathcal{C}}
\newcommand{\diam}{\mathrm{diam}}
\newcommand{\ar}[1]{\overrightarrow{#1}}
\newcommand{\ra}[1]{\overleftarrow{#1}}

\newcommand{\Lhf}{Lipschitz harmonic function}

\begin{document}
	\title{A study of 2-ended graphs via harmonic functions}
	
\author[1]{Agelos Georgakopoulos\thanks{Supported by  EPSRC grants EP/V048821/1 and EP/V009044/1.}}
\affil[1,2]{  {Mathematics Institute, University of Warwick}\\
  {\small CV4 7AL, UK}}
\author[2]{Alex Wendland}

\date{\today}

\maketitle

\begin{abstract}
We prove that every recurrent graph $G$ quasi-isometric to $\mathbb{R}$ admits an essentially unique Lipschitz harmonic function $h$. If $G$ is vertex-transitive, then the action of $Aut(G)$ preserves $\partial h$ up to a sign, a fact that we exploit to prove various combinatorial results about $G$. As a consequence, we prove the 2-ended case of the conjecture of Grimmett \& Li that the connective constant of a non-degenerate vertex-transitive graph is at least the golden mean. 

Moreover, answering a question of Watkins from 1990, we construct a cubic, 2-ended, vertex-transitive graph which is not a Cayley graph.  
\end{abstract}

\medskip



\section{Introduction}

An influential paper of Kleiner \cite{KleNew} provides a proof of Gromov's polynomial growth theorem using harmonic functions on \Cg s. Shalom \& Tao \cite{ShaTaoFin} refined Kleiner's method providing tighter quantitative results, thereby proving that every finitely generated \Cg\ \g admits a non-constant \Lhf, and if \g has polynomial growth, then the vector space of those functions has finite dimension.
These proofs triggered significant interest in the study of \Lhf s on graphs, see \cite{BDKY,BoCoDa,HuJoLiGeo} and references therein. The study of harmonic functions with restrictions on their growth is classical in Riemannian geometry \cite{CheYauDif, ColMinHar, YauNon}, and the origins of the aforementioned proofs can be found there; see \cite{HuaJosPol} for a nice overview. 

Our first result provides the exact dimension of the (real) vector space of \Lhf s on a 2-ended \Cg\ \G, and more generally on graphs that asymptotically resemble \R. Recall that a function $h: V(G) \to \R$ is \defi{harmonic}, if $h(v)$ equals the average of $h(u)$ over the neighbours $u$ of $v$, and it is \defi{Lipschitz}, if $|h(v)-h(u)|<C$ holds for a fixed constant $C$ and every edge $vu\in E(G)$.
\begin{theorem} \label{dim Lip}
Let \g be a \lf, recurrent graph which is quasi-isometric to $\R$. Then the vector space of Lipschitz harmonic functions on \g has dimension equal to 2. 
\end{theorem}
In other words, every such graph admits a non-constant \Lhf, which is unique up to an additive and a multiplicative constant.

We remark that since recurrence is preserved under quasi-isometries between bounded degree graphs \cite[THEOREM~3.2]{CThyperb}, we can replace the recurrence condition by the bounded degree condition in \Tr{dim Lip}. \Tr{dim Lip} probably applies to Riemannian manifolds of bounded geometry that are quasi-isometric to $\R$, with a similar proof. 

The type of harmonic functions supported by a graph or manifold is in general unstable under quasi-isometries \cite{BenIns,LyoIns}, and so it is not easy to deduce \Tr{dim Lip} from the analogous statement for $\R$ or $\Z$. Our proof of \Tr{dim Lip} is not difficult, but uses a mixture of ingredients involving physical intuition around electrical networks, coarse geometry, and the classical fact that a recurrent graph admits no positive, non-constant, harmonic functions. The main message of this paper is perhaps that \Tr{dim Lip} has direct implications for the combinatorial structure of 2-ended vertex-transitive graphs. We say that a graph is \defi{$k$-regular} if all its vertices have degree $k$, and say that it is  \defi{cubic} if it is $3$-regular. As a first application, we use \Tr{dim Lip} to prove

\begin{corollary} \label{cor pm}
Every cubic, 2-ended, vertex-transitive graph \g is 3-edge-colourable. \mymargin{Moreover, it is either bipartite or a \Cg\ generated by 3 involutions???}
\end{corollary}
Since transitive graphs are regular, we immediately deduce from this that the edge set of \g decomposes into three perfect matchings. Cubic graphs are special among the 2-ended, vertex-transitive graphs in that they cannot arise as a cartesian product of a finite graph and a 2-way infinite path, and so understanding the cubic case is an important first step towards understanding all 2-ended vertex-transitive graphs.

We remark that it is a well-known open problem whether every finite, cubic, Cayley graph is $3$-edge-colourable \cite{AlLiZh}. It is known that every infinite, connected, vertex-transitive graph admits a perfect matching \cite{csoka_invariant_2017,Leemann}.

\comment{
\medskip
Using \Tr{dim Lip} again, we will prove
\begin{corollary} \label{cor bip}
Every cubic, 2-ended, vertex-transitive graph is bipartite.
\end{corollary}
This becomes false in the $k$-regular case for $k>3$, a counterexample being the cartesian product $K_{k-2} \times Z$, where $K_n$ stands for the complete graph on $n$ vertices, and $Z$ stands for the 2-way infinite path. }

\medskip
Using \Tr{dim Lip} again, we obtain a simple proof of a theorem of Watkins \cite{WatEdg} saying that every 2-ended edge-transitive graph has even vertex degrees (\Cr{cor ET}).
 
\medskip
Techniques using electrical networks or harmonic functions have found many applications, but usually to the study of  stochastic processes or analytic aspects of graphs; see e.g.\ \cite{planarPB,squareRiemann,GurNachRec,LyonsBook}. Purely combinatorial implications like the above are perhaps more surprising. 

\medskip
The \defi{connective constant $\mu(G)$} of an infinite graph \g is the exponential growth rate of self-avoiding walks in \g starting at a fixed vertex. This invariant is of great interest to the statistical mechanics community, see e.g.\ \cite{DCSmiCon,GriLiLoc} and references therein. As another consequence of \Tr{dim Lip}, we prove
\begin{theorem} \label{muphi}
Every 2-ended vertex-transitive graph \g satisfies $\mu(G)\geq \phi$, where $\phi:=\frac{1+\sqrt{5}}{2}$ is the golden mean.
\end{theorem}
This settles the 2-ended case of a conjecture of Grimmett \& Li \cite{GL16} stating that $\mu(G)\geq \phi$ holds \fe\ infinite, vertex-transitive \G\ other than the 2-way infinite path. Grimmett \& Li \cite{GL16} proved the special case of  \Tr{muphi} where \g is a \Cg, and our approach via harmonic functions closely follows theirs. They thereby revived an old question of Watkins \cite{Wa90}, asking whether \ti\ a cubic, 2-ended, vertex-transitive graph which is not a Cayley graph. In \Sr{sec Watkins} we answer this question by constructing\footnote{This construction was originally announced in the arXiv version of \cite{partite}, but following a referee suggestion it was removed from the journal version.} such a graph:

\begin{theorem} \label{nonCay}
There is a cubic, 2-ended, vertex-transitive graph which is not isomorphic to any \Cg.
\end{theorem}

In \Sr{sec open} we ask if every such graph is bi-Cayley, along with further problems motivated from the above results.

The following folklore open problem emerged naturally from the aforementioned result of Shalom \& Tao \cite{ShaTaoFin}:

\begin{problem} \label{prob Lhf}
Does every \lf\ vertex-transitive graph \g admit a non-constant \Lhf?
\end{problem} 

\Prr{prop exist} below answers this in the affirmative in the (easier) case where \g has at least two ends. The difficult case is where \g has intermediate growth.

\section{Preliminaries}

Let $\g=(V,E)$ be a graph, with vertex set $V$ and edge set $E$. The \defi{degree $d(v)$} of a vertex $v\in V$ is the number of edges incident with $v$. We say that \g is \defi{\lf}, if $d(v)<\infty$ \fe\ $v\in V$. All our graphs below are \lf.

A 1-way infinite path is called a \defi{ray}. Two rays in \g are \defi{equivalent}, if no finite set of edges separates them. The corresponding equivalence classes of rays are the \defi{ends} of \G. We say that an end $\epsilon$ \defi{lives} in a set $X\subseteq V$, if for each ray $R$ in $\epsilon$,  all but finitely many vertices of $R$ lie in $X$.

A \defi{cut} of \g is a bipartition $(X,Y)$ of $V$. We sometimes identify this cut with the set \defi{$E(X,Y):=\{xy\in E \mid x\in X, y\in Y\}$} of edges crossing it. It is not hard to prove that for every two ends $\epsilon,\zeta$ of \G, there is a cut $(X,Y)$ \st\ $\epsilon$ {lives} in a $X$, and $\zeta$ {lives} in a $Y$, and $E(X,Y)$ is finite.

A graph can be thought of as a metric space by endowing it with the graph distance $d$. A \defi{quasi-isometry} between graphs $G=(V,E)$ and $H=(V',E')$ is a map $f: V \to V'$ \st\ the following hold for fixed constants $M\geq 1, A\geq 0$:
\begin{enumerate}
\item $M^{-1} d(x,y) -A \leq d(f(x),f(y))\leq M d(x,y) +A $ \fe\ $x,y \in V$, and
\item \fe\ $z\in V'$ \ti\ $x\in V$ \st\ $d(z,f(x))\leq A$.
\end{enumerate}
We say that $G$ and $H$ are \defi{quasi-isometric}, if such a map $f$ exists. %Quasi-isometries between arbitrary metric spaces are defined analogously.
We will use the following well-known fact:
\begin{proposition}[{\cite[Proposition~2.2]{BriQua}}] \label{prop QI}
Quasi-isometric \lf\ graphs have the same number of ends.
\end{proposition}

Let $\Aut(G)$ denote the group of automorphisms of \G. We say that \g is \defi{vertex-transitive}, if \fe\ $x,y\in V$ \ti\ $z\in \Aut(G)$ \st\ $z(x)=y$. 

\section{Proof of \Tr{dim Lip}}

In this section we prove \Tr{dim Lip}. Let $LH(G)$ denote the real vector space of Lipschitz harmonic functions on a graph \G. We start with the lower bound $\dim(LH(G)) \geq 2$, which we can prove in greater generality as follows:

\begin{proposition} \label{prop exist}
Every connected, \lf\ graph with more than one end admits a non-constant \Lhf.
\end{proposition}

%\begin{proposition} \label{prop unique}
%Let \g be a bounded-degree graph quasi-isometric with $\R$. Then $\dim(LH(G)) =2$.
%\end{proposition}

This provides a positive answer to \Prb{prob Lhf} in the multi-ended case. We will construct the desired harmonic function as a limit of  `electrical potentials' along an increasing sequence of finite subgraphs, corresponding to an electrical current of fixed intensity flowing through a cut separating the ends. We thereby exploit the frequently used viewpoint of harmonic functions as electrical potentials. We start by recalling the relevant electrical network theory.

\subsection{Harmonic functions and electrical network basics}

Let $G=(V,E)$ be a (finite or infinite)  graph. Let $\arE$ denote the set of ordered pairs $(x,y)$ with $xy\in E$. We write $\ar{xy}$ to denote $(x,y)$. Note that for every $xy\in E$, both $\ar{xy}$, $\ar{yx}$ ($=:\ra{xy}$) lie in $\arE$.

We say that a function $f:\arE\rightarrow \mathbb{R}$ is \defi{antisymmetric}, if $f(\ar{xy})=-f(\ra{xy})$ for every $xy\in E$. All functions we consider from now on are antisymmetric.

Given $f:\arE\rightarrow \mathbb{R}$, we write \defi{$f^*(x):=\sum_{y \sim x} f(\ar{xy})$} for the net flow of $f$ out of $x$. Here \defi{$y \sim x$} means that $xy\in E$. 

Fix a graph $G=(V,E)$ for the rest of this section.
\begin{definition} \label{def KNL}
We say that an antisymmetric function $f:\arE \to \mathbb{R}$ satisfies \defi{Kirchhoff's Node Law (KNL)} at a vertex $x$, if $f^*(x)=0$ holds. 

We say that $f$ is a \defi{(sourceless) flow} if it satisfies KNL at every $x\in V$. 
\end{definition}

Intuitively, KNL says that current is preserved at $x$.
Given $p\neq q \in V$, we say that $f:\arE \to \mathbb{R}$ is a \defi{$p-q$~flow}, or a flow from $p$ to $q$, if (KNL) holds for every $x\not \in \{p,q\}$, and $f$ is antisymmetric. The \defi{intensity} of $f$ is defined as $f^*(p)$. If \g is finite, then one can easily show that $f^*(q)=-f^*(p)$ by applying KNL to each other vertex. %Moreover, we say that $f$ is a flow from $p$ if (KNL) holds for every $x\neq p$.

\begin{definition}
We say that a flow $f:\arE \to \mathbb{R}$ satisfies \defi{Kirchhoff's Cycle Law (KCL)}, if for every closed walk (equivalently for every cycle) $x_0,x_1,\ldots,x_n(=x_0)$ in \g we have $\sum_{i=0}^{n-1} f(\ar{x_i x_{i+1}})=0$.
\end{definition}

\begin{definition}
An \defi{(electrical) current} of intensity $I\in\mathbb{R}$ from $p$ to $q$ in $G$, is a $p$-$q$ flow $i:\arE \to \mathbb{R}$ satisfying Kirchhoff's Cycle Law, \st\ $i^*(p)=I=-i^*(q)$. %If $i$ has intensity 1, then it is called a \defi{unit current}.
\end{definition}

We say that a pair of functions $i:\arE \to \mathbb{R}$, $u:V\rightarrow \mathbb{R}$ satisfies \defi{Ohm's Law (OL)},  if $i(\ar{xy})=u(x)-u(y)$ holds \fe\ $x,y\in V$.

Note that if $i,u$ satisfies (OL), then so does $i,(u+a)$ for every constant $a\in\mathbb{R}$. We say that $i$ is the \defi{Ohm dual} of $u$ and $u$ is an Ohm dual of $i$, if (OL) holds.

\begin{definition} \label{def harm}
We say that a function $u: V \to \R$ is \defi{harmonic}, if 
$u(x) =\sum_{y\sim x} u(y)/d(x)$ holds \fe\ $x\in V$. 
\end{definition}

\begin{proposition} \label{prop knl}
If the pair $i,u$ satisfies Ohm's Law, then $u$ is harmonic at $x\in V$ if and only if $i$ satisfies Kirchhoff's Node Law at $x$, i.e.\ if $i^*(x)=0$.
\end{proposition}

This is well-known, but we include the proof for convenience as it is quite short:
\begin{proof}
Since $i,u$ satisfy (OL), we have 
$$\sum_{y\sim x} u(y)=\sum_{y\sim x} \big( u(x)-i(\ar{xy}) \big)=d(x)u(x)-\sum_{y \sim x} i(\ar{xy}),$$
which is by definition equal to $d(x)u(x)-i^*(x)$. Thus $i^*(x)=0$ if and only if $\sum_{y \sim x} u(y)/d(x)=u(x)$, which is the definition of being harmonic.
\end{proof}

\begin{proposition}[{The maximum/minimum principle \cite[Lemma 3.2]{ma3h2}}] \label{max prin}
Suppose \g is finite and connected, and $u:V \to \R$ is harmonic at every vertex of $V'\subset V$. If the maximum or minimum of $u$ is achieved at some vertex of $V'$, then $u$ is constant.
\end{proposition}

Note that every function $u:V \to \R$ has a (unique) Ohm dual \defi{$\partial u$} defined by $\partial u(\ar{xy})\mapsto u(x)-u(y)$. For the converse we have the following well-known fact: \begin{proposition}[{\cite[Proposition 3.7.]{ma3h2}}]  \label{OD}
Let $G$ be a connected graph, and $i: \arE \to \mathbb{R}$ antisymetric. Then, there is $u: V\rightarrow \mathbb{R}$ such that the pair $i,u$ satisfies OL if and only if $i$ satisfies KCL. This $u$ is unique up to an additive constant.
\end{proposition}

Perhaps better-known among electrical engineers, is the fact that no branch of an electrical network \g can enhance the current flowing into $G$ from the external source: 
\begin{proposition} \label{prop leq1}
Let $G=(V,E)$ be a finite graph, and $p,q\in V$. Let $f:\arE\to \R$ be the $p$--$q$~electrical current of intensity 1 in \G. Then $|f(\ar{wv})|\leq 1$ holds \fe\ $\ar{wv} \in \arE$.
\end{proposition}
%%%%%%%%%
\begin{proof}
Let $h$ be an Ohm dual of $f$ as provided by \Prr{OD}. Given an edge $wv\in E$, let 
$$A:= \{x\in V \mid h(x) \geq \frac{h(w) + h(v)}{2}\}$$
and $Z:= V \sm A$. Thus the cut $(A,Z)$ separates $w,v$ whenever $f(\ar{wv})\neq 0$, which we may assume is the case. Note that 
\labtequ{fxz}{$f(\ar{xz})>0$ holds \fe\ $x\in A, z\in Z$. }
Moreover, we have $\sum_{x\in A} f^*(x) \leq 1$ since $f$ is a  current of unit intensity. By double counting, we have $\sum_{x\in A} f^*(x) = \sum_{x\in A, z\in Z} f(\ar{xz}) \geq |f(\ar{wv})|$, where the last inequality follows from \eqref{fxz} and the fact that $(A,Z)$ separates $w,v$. Combining these two inequalities we obtain the desired $|f(\ar{wv})| \leq 1$.
\end{proof}

\subsection{Existence of \Lhf s}

We have now gathered the necessary ingredients to prove our existence result:
%%%%%%%%%
\begin{proof}[Proof of \Prr{prop exist}]
Assume first that \g has exactly two ends $\epsilon,\zeta$; the general case is similar. 

Fix an \defi{origin} $o\in V(G)$, and consider the \defi{spheres}  $S_n:=\{v\in V(G) \mid d(o,v)=n\}$. The \defi{ball} $B_n$ is the subgraph of \g induced by the vertex set $\{v\in V(G) \mid d(o,v)\leq n\}$. Define, for each $\nin$, the \defi{layer} $L_n$ (respectively $L_{-n}$) 
as the subset of $S_n$ contained in the unique component of $G - B_{n-1}$ containing a ray of $\epsilon$ (resp.\ $\zeta$). Note that $L_n$ coincides with $L_{-n}$ as long as $B_{n-1}$ does not separate the two ends, but $L_n, L_{-n}$ are disjoint for $n>n_0:= \min \{k \mid B_k \text{ separates } \epsilon,\zeta \}$.

For each $n> n_0$, we define a flow $f_n$ in $B_n$ as follows. Pick vertices $p\in L_n, q\in L_{-n}$, and let $f_n$ be the electrical current of intensity 1 from $p$ to $q$ in  $B_n$. %Perform a simple random walk $X$ on $B_n$, starting at some vertex $x\in L_n$, and stopping upon its first visit to $L_{-n}$ (which always happens, as $B_n$ is finite). For each directed edge $e = vw$ of $B_n$, let 
%$$f_n():= \Ex{ \text{ number of times $X$ traverses $e$ from $v$ to $w$}} - \Ex{ \text{ number of times $X$ traverses $e$ from $w$ to $v$}}.$$

Moreover, for every cut $C=E(X,Y)$ of $B_n$ separating $L_n$ from $L_{-n}$, the net flow of $f_n$ through $C$ equals 1 because the net flow of $f_n$ out of $L_n$ is 1, and every vertex outside $L_n \cup L_{-n}$ preserves the net flow. Fix such a cut $C$ in $B_{n_0+1}$, and note that $C$ separates $\epsilon$ from $\zeta$ in \G, and therefore it  separates $L_i$ from $L_{-i}$ \fe\ $i>n_0$. %, because our random walker always starts in $L_n$ and ends up in $L_{-n}$. 

Let $f: \arE(G) \to \R$ be a pointwise accumulation point of the sequence \seq{f}, which exists since the space $[-1,1]^{\arE(G)}$ in which our functions live is compact by Tychonoff's theorem. 

Then $f$ is not identically 0, because its net flow through the finite cut $C$ equals 1. Moreover, $f$ satisfies \kcl\ for every cycle $K$ of \G, since every $f_n$ does. Thus we can define the Ohm dual  $h$ of $f$ satisfying $h(o)=0$ by \Prr{OD}.  Finally, $f$ satisfies \knl\ at every vertex $v\in V(G)$, since almost every $f_n$ does. Therefore $h$ is a harmonic function by \Prr{prop knl}. It is non-constant since  $f$ is not identically 0. By \Prr{prop leq1}, we have $|f_n(\are)|\leq 1$ \fe\ $\are \in \arE$, and hence $|f(\are)|\leq 1$, which means that $h$ is Lipschitz.

\medskip
If \g has more than two ends, we can repeat the construction using an appropriate bipartition of the ends as follows. We let  $C$ be a minimal cut separating two infinite components $K_1,K_2$. Let $n_0:= \min \{k \mid C\subset B_k \}$. 
For  $n> n_0$, define  $L_n:= S_n \cap K_1$ and $L_{-n}:= S_n \cap K_2$. The rest of the construction of $h$ remains the same.
\end{proof}

\subsection{Upper bounding $\dim(LH(G))$} \label{sec UB}

Our proof of \Tr{dim Lip} will come as a combination of the above and the following lemmas. Given a cut $(X,Y)$ of a graph \G, and a flow $f: \arE(G) \to \R$, we define \defi{$f(X,Y):= \sum_{x\in X, y\in Y} f(\ar{xy})$}.

\begin{lemma} \label{lem cuts}
Let $G=(V,E)$ be a connected, 2-ended graph.
For every sourceless flow  $f:\arE\to \R$, and every two finite cuts $(X,Y)$ and $(X',Y')$ separating the ends, we have 
$f(X,Y)=\sigma f(X',Y')$, where $\sigma=1$ if $X,X'$ contain the same end, and $\sigma=-1$ otherwise. 
\end{lemma}
%%%%%%%%%
\begin{proof}
Let $\eps$ be the end of \g that lives in $X$, and $\zeta$ the end that lives in $Y$. We claim that if $\eps$ lives in $X'$, then the symmetric difference $F= X \sydi X' := (X \sm X') \cup (X' \sm X)$ is finite. For if not, then one of $X,X'$ contains an infinite sequence $\{v_i\}$ of vertices not contained in the other, and as $v_i$ must converge to $\eps$ we have a contradiction. 

Next, we claim that $f(X',Y') = f(X,Y) + \sum_{v\in F} f^*(v)$. This is easy to check by induction on the size of $F$, because the effect on $f(X,Y)$ of moving a vertex $v$ from $Y$ to $X$ is exactly the addition of $f^*(v)$. Since $f^*(v)=0$ \fe\ $v$ as $f$ is a sourceless flow, we deduce $f(X,Y)= f(X',Y')$. 

If $\eps$ lives in $Y'$ instead, then we repeat the same arguments with the roles of $X'$ and $Y'$ interchanged, to deduce that $f(X,Y)= f(Y',X') = - f(X',Y')$.
\end{proof}

\begin{lemma} \label{lem const}
Let $G=(V,E)$ be a \lf, recurrent graph which is quasi-isometric to $\R$. Let $h$ be a \Lhf\ on \G, and suppose that for some finite cut $E(X,Y)$ separating the two ends of \g we have $\partial h(X,Y)= 0$. Then $h$ is constant. 
\end{lemma}
%%%%%%%%%
\begin{proof}
By \Prr{prop QI}, $G$ has exactly two ends $\epsilon,\zeta$.

It is well-known that every recurrent graph has no non-constant positive harmonic functions \cite[Theorem~(1.16)]{woessBook}\footnote{The analogous statement for Riemannian manifolds is also true \cite[Theorem 5.1]{GriAna}.}. It follows from this that if $h$ is non-constant, then it is unbounded both above and below. Thus \ti\ a sequence \seq{v}\ of vertices \st\ $h(v_n) \to \infty$. Some subsequence of \seq{v}\ converges to one of the ends, and so we may assume \obda\ that $v_n \to \epsilon$. 

Next, we claim that 
\labtequ{eps}{for every sequence \seq{w}\ of vertices of \g that converges to $\epsilon$, we have $h(w_n) \to \infty$. }
To see this, given $v\in V$, let $S_v \subset G$ be a connected subgraph containing $v$ and separating $\epsilon$ from $\zeta$. Using a quasi-isometry $q: G \to \R$, it is not hard to choose the $S_v$ to have diameters $\diam(S_v)<M$ for a constant $M$ depending on $q$ but not on $v$. 

Since $h$ is Lipschitz, it follows that if $x,y\in S_v$, then $|h(x)-h(y)|<M' \in \R$. Thus 
\labtequ{hS}{$h[S_{v_n}] \to \infty$ for the above sequence \seq{v}.}

Next, let $H'_i$ denote the union of the finite components of $G - (S_{v_i} \cup S_{v_{i+1}})$, and let $H_i:= H'_i \cup S_{v_i} \cup S_{v_{i+1}}$. Notice that $H_i$ is a finite graph. Thus \eqref{hS}, combined with the maximum principle (\Prr{max prin}), implies that  $h[H_n] \to \infty$. Since $\seq{v},\seq{w}$ both converge to $\epsilon$, almost every $w_n$ lies in some $H_{i(n)}$ with $i(n) \to \infty$ as $n\to  \infty$. The last two remarks combined imply $h(w_n)\to \infty$ as claimed.

\medskip
Since $h$ is also not bounded below, we obtain a sequence  \seq{z}\ of vertices \st\ $h(z_n) \to -\infty$. Since no subsequence of \seq{z}\ can converge to $\epsilon$ by \eqref{eps}, it follows that $z_n \to \zeta$. By repeating the proof of  \eqref{eps} we deduce that $h(y_n) \to -\infty$ holds whenever $y_n \to \zeta$. 

Let $P:= \{x\in V(G) \mid h(x)\geq 0\}$, and $N:= V\sm P= \{x\in V(G) \mid h(x) < 0\}$. Note that the cut $E(P,N)$ is non-empty. We claim that $E(P,N)$ consists of finitely many edges. For if ${x_n}$ is a sequence of distinct end-vertices of edges in $E(P,N)$ lying in $P$, then some subsequence will converge to $\epsilon$ or $\zeta$. We would then have $|h(x_n)| \to \infty$ by \eqref{eps} and its analogue for $\zeta$. But this is impossible: since $h$ is Lipschitz, and $x_n$ has a neighbour in $N$, the values $h(x_n)$ must be uniformly bounded.

Recall that we are assuming $\partial h(X,Y)= 0$, and therefore \Lr{lem cuts} implies $\partial h(P,N)= 0$. This is impossible, because $E(P,N)$ is non-empty, and $\partial h(x,y)>0$ holds \fe\ $xy\in E(P,N)$ with $x\in P, y\in N$. 
This contradiction proves that $h$ is constant as claimed.
\end{proof}

\subsection{Completing the proof of \Tr{dim Lip}}
Our main result now follows as a direct combination of \Prr{prop exist} and \Lr{lem const}: 

\begin{proof}[Proof of \Tr{dim Lip}]
By \Prr{prop QI}, $G$ has exactly two ends $\epsilon,\zeta$.
To see that $\dim(LH(G)) \geq 2$, let $f_1 \equiv 1$, and let $f_2$ be a non-constant \Lhf\ of \g as provided by \Prr{prop exist}. Then $f_1,f_2$ clearly span a 2-dimensional subspace of $LH(G)$. 
\medskip

%...We claim that $W=LH(G)$...
For the converse inequality, fix a finite cut $E(X,Y)$ of \g separating $\epsilon,\zeta$. We claim that if two functions $h_1,h_2\in LH(G)$ satisfy $\partial h_1(X,Y)=\partial h_2(X,Y)$ and $h_1(o)=h_2(o)$ for some $o\in V(G)$, then $h_1 = h_2$ holds. This clearly implies that $\dim(LH(G)) \leq 2$.

To prove this claim, let us consider the difference $h:= h_1 - h_2 \in LH(G)$. Notice that $\partial h(X,Y)= 0$. Thus \Lr{lem const} says that $h$ is constant, and since $h_1(o)=h_2(o)$ we deduce that $h \equiv 0$ and so $h_1=h_2$.
\end{proof}

Let us explore the tightness of the conditions of the results of this section.

\Prr{prop exist} fails if \g is 1-ended, as can be easily checked by letting \g be a 1-way infinite path. %But the proof is not hard to generalise to the case where \g has more than 2 ends.

\Tr{dim Lip} (and \Lr{lem const}) becomes false if we drop the recurrence condition. To see this, start with the 2-way infinite ladder $L$, i.e.\ the graph obtained from the disjoint union of two 2-way infinite paths $H_1,H_2$ by joining each vertex of $H_1$ to the corresponding vertex of  $H_2$ with an edge.  We call the edges in each $H_i$ \defi{horizontal}, and the remaining edges \defi{vertical}. Obtain \g by replacing each horizontal edge of $L$ at distance $n$ from the origin by $2^n$ paths of length 2. The interested reader will be able to construct a bounded harmonic function $h$ that converges to four different values as we move along the four horizontal rays. One way of doing so is by noticing that simple random walk on \g will almost surely traverse only finitely many vertical edges, and defining $h(v)$ as the probability for random walk starting at $v$ to eventually stick to a fixed horizontal ray.

\Tr{dim Lip} and \Lr{lem const} easily become false if we drop the quasi-isometry condition, even if we assume 2-endedness: let $H$ be an 1-ended graph that admits non-constant bounded harmonic functions, and connect two copies of $H$ by one edge.

If we drop  the Lipschitz condition in \Lr{lem const} then we do get non-constant harmonic functions, but the only examples I know grow exponentially in the distance from a fixed origin. I suspect this cannot be improved:

\begin{conjecture}
Let $G=(V,E)$ be a \lf, recurrent graph which is quasi-isometric to $\R$. Let $h$ be a harmonic function on \G, and suppose that for some finite cut $E(X,Y)$ separating the two ends of \g we have $\partial h(X,Y)= 0$. Then $h$ is either constant or grows exponentially, i.e.\ there is a constant $c>1$ and a sequence of vertices \seq{v}\ \st\ $|h(v_n)-h(v_0)|>c^{d(v_n,v_0)}$. 
\end{conjecture}

\Prr{prop exist} shows that the condition $\partial h(X,Y)= 0$ is crucial here.

\section{Consequences for vertex-transitive graphs} \label{sec VT}

In this section we employ \Tr{dim Lip} to obtain results about 2-ended vertex-transitive graphs as mentioned in the introduction.

\subsection{Connective constant}

We now prove  \Tr{muphi}, which provides a lower bound on the connective constant $\mu(G)$ of a 2-ended, cubic, vertex-transitive graph. We will refrain from recalling the definition of $\mu(G)$ as we will never use it, and refer the interested reader to \cite{GL16}.
%%%%%%%%%
\begin{proof}[Proof of \Tr{muphi}]
We follow the approach of Grimmett \& Li \cite{GL16}, who proved the statement for a cubic \Cg\ \G. The case of degree at least 4 had been settled earlier by the same authors \cite{GriLiBou}.   Their proof starts by constructing a non-constant harmonic function $g$ of $G$, with the property that $g$ is \defi{skew-difference-invariant}, i.e.\ \fe\ $u,v\in V(G)$, and $z\in \Aut(G)$, we have  
\labtequ{skew}{$g(z(v))-g(z(u)) = \sigma(z)(g(v)-g(u)),$}
where $\sigma(z)\in \{-1,1\}$. The rest of their proof (starting at the top of p.~34 of \cite{GL16}) does not rely on the structure of \g and can be repeated verbatim. Thus it suffices to construct such a function $g$.

For this, let $g: V(G) \to \R$ be a non-constant, Lipschitz, harmonic function, provided by \Tr{dim Lip}. Given any automorphism $z\in \Aut(G)$, the function $g \circ z$ is also in $LH(G)$. By \Tr{dim Lip}, $g \circ z$ equals $g$ up to a multiplicative constant $c_z\in \R$ and an additive constant $a_z$. Note that if $F$ is a finite cut of \g separating the two ends $\epsilon,\zeta$, then the net flow of $\partial g$ from the component of $G-F$ containing $\epsilon$ to the component of $G-F$ containing $\zeta$ is independent of the choice of $F$ up to a sign by \Lr{lem cuts}. It follows that $c_z=\pm 1$. Thus letting $\sigma(z):=c_z$ satisfies \eqref{skew}. 
\end{proof}

\subsection{No edge-transitive graphs with odd degrees}

By adapting our last argument we can recover the following result of Watkins, which is the main result of \cite{WatEdg}.

\begin{corollary}[\cite{WatEdg}] \label{cor ET}
If \g is a 2-ended, edge-transitive graph, then every vertex of \g has even degree. 
\end{corollary}
%%%%%%%%%
\begin{proof}
Suppose $o\in V(G)$ has odd degree, and let 
again $g: V(G) \to \R$ be a non-constant, Lipschitz, harmonic function, provided by \Tr{dim Lip}. Let $ou,ov$ be two edges of $o$ \st\ $|g(o)-g(u)| \neq |g(o)-g(v)|$, which exist because  by \Prr{prop knl} the net flow of $\partial g$ out of $o$ is 0, and $o$ has odd degree.

Since \g is edge-transitive, there is an automorphism $z\in \Aut(G)$ \st\ $z(ou) = ov$. Note that the function $g':=g \circ z$ is also in $LH(G)$. By \Tr{dim Lip}, $g \circ z$ equals $g$ up to a multiplicative constant $c_z\in \R$ and an additive constant $a_z$. Pick a finite cut $E(X,Y)$ separating the ends of \G. By \Lr{lem cuts} we have $\partial g(X,Y)=\pm  \partial g'(X,Y)$ because $\partial g'(X,Y) = \partial g(z(X),z(Y))$. Since $a_z$ does not influence $\partial g'$, we deduce $c_z = \pm 1$. But this contradicts our assumption $|g(o)-g(u)| \neq |g(o)-g(v)|$. 
\end{proof}

\subsection{Colourings}
%%%%%%%%%
\begin{proof}[Proof of \Cr{cor pm}]
Let $g\in LH(G)$ be as in the proof of \Tr{muphi}, that is,  non-constant and skew-difference-invariant. Repeating an idea of Grimmett \& Li \cite{GL16}, we let $a\leq b \leq c$ denote the values of $g(v)-g(o)$ for $v$ a neighbour of $o\in V(G)$, and note that since $a+b+c = 0$ by \Prr{prop knl}, one of the following three cases must occur if we rescale $g$ appropriately:
\begin{itemize}
	\item[Case 1:] \label{O i}  $(a,b,c) = (-1,0,1)$. In this case, colour all edges $uv$ with $g(u)-g(v)=0$ red. Notice that the remaining edges span a 2-regular subgraph $H$ of $G$. Each component of $H$ is a 2-way infinite path or a cycle. Since each edge $uv$ of such a cycle $C$ satisfies $|g(u)-g(v)|=1$, and the sum of $g(u)-g(v)$ along $C$ equals 0, we deduce that $C$ must contain an even number of edges. Thus we can colour each component of $H$ blue-green in an alternating way.
	\item[Case 2:] \label{O ii} $(a, b, c) = (-1/2,-1/2,1)$. Colour all edges $uv$ with $|g(u)-g(v)|=1$ red, and repeat the argument of Case~1 to colour the remaining edges blue-green.
	\item[Case 3:] \label{O iii} $a< b < 0, -a-b = c$. This case is easy: we can just colour each edge $uv$ according to $|g(u)-g(v)|$.
\end{itemize}
\end{proof}

\comment{
\subsection{bipartite} \label{sec bip}

%%%%%%%%%
\begin{proof}[Proof of \Cr{cor bip}]
We distinguish the same three cases as in the proof of \Cr{cor pm}.
\end{proof}
}

\section{A cubic 2-ended vertex-transitive graph which is not Cayley} \label{sec Watkins}

Watkins \cite{Wa90} asked whether every 2-ended vertex-transitive graph can be obtained via a certain product operation between a finite vertex-transitive graph and the 2-way infinite path. A positive answer to Watkins' question would imply that every 2-ended cubic vertex-transitive graph is a Cayley graph. Whether this is true was more recently explicitly asked by Grimmett \& Li \cite{GL16} due to its relevance to \Tr{muphi}. In this section we construct an example of a 2-ended cubic vertex-transitive graph which is not a Cayley graph, thus providing a negative answer to these   questions.

\medskip
Let $\Gha$ be the graph with $V(\Gha) = \{v_{n,k} \vert n \in \bZ, k \in \bZ/10\bZ\}$ and
\[
 E(\Gha) = \{v_{n,k} v_{n,k+1}\ \vert \ n \in \bZ, \ k \in \bZ/10\bZ\} \cup \{ v_{n,2k+1} v_{n+1,4k+2} \ \vert \ n \in \bZ, \ k \in \bZ/10\bZ\}.
\]
By construction, $\Gha$ is a cubic graph. A useful way to think of this graph is as a 2-way infinite stack of layers  $L_n := \{v_{n,k} \vert k \in \bZ/10\bZ\}$. Each $L_n$ spans a 10-cycle, and between any two layers $L_n$ and $L_{n+1}$ there is a Petersen-graph-like structure. Thus $\Gha$ is clearly 2-ended, since contracting each $L_n$ into a vertex transforms $\Gha$ into a 2-way infinite path. Checking that $\Gha$ has all other desired properties is rather routine, but we include the details for the sake of completeness. 

\begin{cla}
	$\Gha$ is vertex-transitive.
\end{cla}
\begin{proof}

We introduce the following two maps $\sigma, \tau : V(\Gha) \to V(\Gha)$, which will allow us to map any vertex of $\Gha$ to any other:
\[
\sigma(v_{n,k}) = v_{n+1,k} \ \mbox{ and } \ \tau(v_{n,k}) = \begin{cases} v_{-n, k+1} & \mbox{ if } n \equiv 0 \mbox{ (mod }4\mbox{)}\\ v_{-n, 3-k} & \mbox{ if } n \equiv 1 \mbox{ (mod }4\mbox{)}\\v_{-n, k+9} & \mbox{ if } n \equiv 2 \mbox{ (mod }4\mbox{)}\\v_{-n, 7-k} & \mbox{ if } n \equiv 3 \mbox{ (mod }4\mbox{)} \end{cases}.
\] 
Thus $\sigma$ just shifts all the layers up by 1, whereas $\tau$ rotates the 10-cycle on $L_0$ by one position, which flips the stack of layers by mapping $L_n$ to $L_{-n}$, and inverts the orientation of the 10-cycles at layers of odd index.  
It is easy to see that each of $\tau$ and $\sigma$ is a bijection preserving edges, and so $\sigma, \tau \in \Aut(\Gha)$. Let 
$\Gpa := \langle \sigma, \tau \rangle \leq Aut(\Gha)$
be the group of automorphisms of $\Gha$ generated by  $ \sigma$ and $\tau$.
For any two vertices $v_{n,k}, v_{n',k'} \in V(\Gha)$, we have 
\begin{align*}
\sigma^{n'}\tau^{k'-k}\sigma^{-n}(v_{n,k}) = & \sigma^{n'}\tau^{k'-k}(v_{0,k})\\
= & \sigma^{n'}(v_{0,k'})\\
= & v_{n',k'}
\end{align*}
and so $\Gpa$ acts transitively on $V(\Gha)$. This proves Claim~1.
\end{proof}

A classical theorem of Sabidussi \cite{sab} states that a graph \g is isomorphic to a \Cg\ \iff\ there is a subgroup of $\Aut(G)$ acting \defi{regularly}, i.e.\ freely and transitively, on $V(G)$. For the purposes of this paper, they reader may use Sabidussi's theorem as the definition of a \defi{\Cg}. We will use this to prove that our $\Gha$ is not a \Cg. 

\begin{cla} \label{not-regular}
	$\Gpa$ is not a regular subgroup of $\Aut(\Gha)$.
\end{cla}
\begin{proof}

We have
\begin{align*}
\tau^{-3}\sigma\tau\sigma (v_{0,k}) & = \tau^{-3} \sigma \tau (v_{1,k})\\
& = \tau^{-3} \sigma (v_{-1, 3 - k})\\ & = \tau^{-3} (v_{0,3 - k})\\ & = v_{0,- k}.
\end{align*}
This implies $\tau^{-3}\sigma\tau\sigma(v_{0,0}) = v_{0,0}$, yet $\tau^{-3}\sigma\tau\sigma \not = 1_{\Gha}$ as $\tau^{-3}\sigma\tau\sigma(v_{0,1}) = v_{0,9}$. Thus the action of $G$ on $\Gha$ is not free, proving Claim~2.
\end{proof}

\medskip
We remind the reader the $n$th layer is denoted $L_n = \{v_{n,k} \vert \ k \in \bZ/10\bZ \}$ for $n \in \bZ$ and we define the partition $\sC := \{L_n\ \vert\ n \in \bZ\}$ of $V(\Gha)$.

\begin{cla} \label{fix-one-cycle}
	Let $\phi \in \Aut(\Gha)$ satisfy $\phi(L_a) = L_b$. If for some $\chi \in \Aut(\Gha)$ we have $\phi(x)=\chi(x)$ for every $x\in L_a$,  then  $\phi=\chi$. Moreover, $\phi$ preserves the partition $\sC := \{L_n\ \vert\ n \in \bZ\}$.
\end{cla}
\begin{proof}

Observe that $\phi(v_{a+1,2k})$ and $\phi(v_{a-1,2k+1})$ are uniquely determined by $\phi(x), x\in L_a$ because each vertex in $L_b$ has exactly one neighbour outside $L_b$. This in turn uniquely determines $\phi(v_{a+1,2k+1})$ and $\phi(v_{a-1,2k})$ by a similar argument. Continuing like this, we see that $\phi(\{v_{a + \epsilon,k} \vert k \in \bZ/10\bZ\}) = \{v_{b \pm \epsilon,k} \vert k \in \bZ/10\bZ\}$ for $\epsilon \in \{-1,1\}$. By an inductive argument, this uniquely determines $\phi$, and moreover $\phi$ preserves $\sC$.
\end{proof}

\begin{cla} \label{cycle-pres}
	Any $\phi \in \Aut(\Gha)$ preserves the partition $\sC = \{L_n \vert n \in \bZ\}$.
\end{cla}
\begin{proof}
 
If $\phi$ does not fix $\sC$, then by Claim \ref{fix-one-cycle} we have $\phi(L_0) \not = L_a$ for every $a \in \bZ$. It is not hard to see that no 10-cycle of $\Gha$ other than the $L_i$ separates $\Gha$ into two infinite components. This contradicts that $\phi$ is an automorphism. (An alternative way of proving the claim is by using Watkins's theorem that every 2-ended vertex-transitive and edge-transitive graph has even degree \cite{Wa90}.)
\end{proof}

\begin{cla} \label{unique-determined}
	Any $\phi \in \Aut(\Gha)$ is uniquely determined by $\phi(v_{0,1})$ and $\phi(v_{0,2})$.
\end{cla}
\begin{proof}

Assume $\phi(v_{0,1}) = v_{a,b}$. Then by Claim \ref{cycle-pres}, $\phi(v_{0,2}) \in N(v_{a,b}) \cap \{v_{a,k} \vert k \in \bZ/10\bZ\} = \{v_{a,b+1},v_{a,b-1}\}$. In either case, by Claim \ref{cycle-pres} this uniquely determines $\phi(v_{0,k})$ for $k \in \bZ/10\bZ$. By Claim \ref{fix-one-cycle}, this uniquely determines $\phi$.
\end{proof}

We remark that this implies $\mbox{Stab}_{\Aut(\Gha)}(v_{0,0}) = \langle \tau^{-3}\sigma\tau\sigma \rangle = \bZ/2\bZ$. So $\langle \tau, \sigma \rangle = \Aut(\Gamma)$.

\begin{cla}
	$\Gha$ is not isomorphic to a Cayley graph.
\end{cla}
\begin{proof}

Let $T \leq \Aut(\Gha)$ be a transitive subgroup. Thus we can find automorphisms $\sigma', \tau' \in T$ such that $\sigma'(v_{0,1}) = v_{1,1}$ and $\tau'(v_{0,1}) = v_{0,2}$. By Claim \ref{cycle-pres}, $T$ preserves the partition $\sC$. So either $\tau'(v_{0,2}) = v_{0,3}$ and $\tau' = \tau$ or $\tau'(v_{0,2}) = v_{0,1}$ and  
\[
\tau'(v_{n,k}) = \sigma \tau \sigma(v_{n,k}) = \tilde{\tau}(v_{n,k}) := \begin{cases} v_{-n, 3 - k} & \mbox{ if } n \equiv 0 \mbox{ (mod }4\mbox{)}\\ v_{-n, k-1} & \mbox{ if } n \equiv 1 \mbox{ (mod }4\mbox{)}\\v_{-n, 7-k} & \mbox{ if } n \equiv 2 \mbox{ (mod }4\mbox{)}\\v_{-n, k+1} & \mbox{ if } n \equiv 3 \mbox{ (mod }4\mbox{)} \end{cases}
\] 
by Claim \ref{unique-determined}. Similarly, either $\sigma'(v_{0,2}) = v_{1,2}$ and $\sigma' = \sigma$ or $\sigma'(v_{0,2}) = v_{1,0}$ and
\[
\sigma'(v_{n,k}) = \sigma \tau^{-1} \sigma \tau \sigma(v_{n,k}) = \tilde{\sigma}(v_{n,k}) := \begin{cases} v_{n+1, 2-k} & \mbox{ if } n \equiv 0 \mbox{ (mod }4\mbox{)}\\ v_{n+1, 4-k} & \mbox{ if } n \equiv 1 \mbox{ (mod }4\mbox{)}\\v_{n+1, 8-k} & \mbox{ if } n \equiv 2 \mbox{ (mod }4\mbox{)}\\v_{n+1, 6-k} & \mbox{ if } n \equiv 3 \mbox{ (mod }4\mbox{)} \end{cases}.
\] 
by Claim \ref{unique-determined}. By Claim \ref{not-regular}, if $\{\sigma', \tau'\} = \{\sigma, \tau\}$ then $T$ is not regular. If $\{\sigma', \tau'\} = \{\tilde{\sigma}, \tau\}$ then
\begin{align*}
\tau\tilde{\sigma}\tau\tilde{\sigma} (v_{0,k}) & = \tau \tilde{\sigma} \tau (v_{1,2-k})\\
& = \tau \tilde{\sigma} (v_{-1, k + 1})\\ 
& = \tau (v_{0,5 - k})\\ 
& = v_{0,6 - k}
\end{align*}
giving $\tau\tilde{\sigma}\tau\tilde{\sigma}(v_{0,3}) = v_{0,3}$ yet $\tau\tilde{\sigma}\tau\tilde{\sigma} \not = 1_{\Gha}$. Similarly, if $\{\sigma', \tau'\} = \{\sigma, \tilde{\tau}\}$ then
\begin{align*}
\tilde{\tau}\sigma\tilde{\tau}\sigma (v_{0,k}) & = \tilde{\tau}\sigma\tilde{\tau} (v_{1,k})\\
& = \tilde{\tau}\sigma (v_{-1, k-1})\\ 
& = \tilde{\tau} (v_{0,k-1})\\ 
& = v_{0,4 - k}
\end{align*}
giving $\tilde{\tau}\sigma\tilde{\tau}\sigma(v_{0,2}) = v_{0,2}$ yet $\tilde{\tau}\sigma\tilde{\tau}\sigma \not = 1_{\Gha}$. Lastly, if $\{\sigma', \tau'\} = \{\tilde{\sigma}, \tilde{\tau}\}$ then 
\begin{align*}
\tilde{\tau}\tilde{\sigma}\tilde{\tau}\tilde{\sigma} (v_{0,k}) & = \tilde{\tau} \tilde{\sigma} \tilde{\tau} (v_{1,2-k})\\
& = \tilde{\tau} \tilde{\sigma} (v_{-1, 1-k})\\ 
& = \tilde{\tau} (v_{0,5 + k})\\ 
& = v_{0,-2 - k}
\end{align*}
giving $\tilde{\tau}\tilde{\sigma}\tilde{\tau}\tilde{\sigma}(v_{0,9}) = v_{0,9}$ yet $\tilde{\tau}\tilde{\sigma}\tilde{\tau}\tilde{\sigma} \not = 1_{\Gha}$. Therefore $T$ is never regular, and so $\Gha$ is not a Cayley graph by Sabidussi's theorem mentioned above.
\end{proof}

\medskip
Combining the above claims proves \Tr{nonCay}. %we deduce %that $\Gha$ is a cubic 2-ended vertex-transitive graph which is not Cayley, so we have proved the following.

\comment{
	\begin{theorem} 
	$\Gha$ is a cubic 2-ended vertex-transitive graph which is not a Cayley graph.

	\qed
	\end{theorem} 
}

We have implicitly proved that $\Aut(\Gamma) = \langle \sigma, \tau \rangle$. Some further study, for which we thank Derek Holt, can show that \[ \Aut(\Gamma) = \langle  \sigma, \tau \vert \tau^{10}, (\tau^{-1} \sigma \tau \sigma)^2, \sigma^{-1} \tau^2 \sigma\tau^{-4}, (\sigma^{-2}\tau)^2 \rangle. \]

\section{Further problems} \label{sec open}
In this paper we have used harmonic functions to deduce various facts about 2-ended, vertex-transitive graphs. We feel that the technique can be exploited further to yield a better understanding of this class of graphs, possibly leading to a complete description. Towards this aim, we propose

\begin{problem} \label{prob eff}
Provide an effective enumeration of the family of 2-ended, \lf, vertex-transitive graphs.
\end{problem} 

By an \defi{effective enumeration} of a family \cf\ we mean an algorithm that outputs a sequence of finite strings \seq{S}, \st\ each $S_i$ \defi{encodes} an element of \cf, and each element of \cf\ is encoded by some $S_i$. We do not prescribe a specific way to encode our graphs, but one option is to use our partite presentations from \cite{partite}, which is a generalisation of group presentations to arbitrary vertex-transitive graphs. It would be already interesting to provide an effective enumeration of the cubic, 2-ended, vertex-transitive graphs 

\medskip
A vertex-transitive graph \g is said to be $n$-Cayley, if there is a subgroup $\Gamma$ of $\Aut(G)$ \st\ $\Gamma$ acts freely (i.e.\ without fixing any point) on $V(G)$ and has $n$ orbits of vertices. Thus \g is isomorphic to a \Cg\ \iff\ it is 1-Cayley by Sabidussi's aforementioned theorem. If $n = 2$ we say that \g is \defi{bi-Cayley}.

\begin{problem} \label{prob biC}
Is every 2-ended, cubic, vertex-transitive graph Cayley or bi-Cayley?
\end{problem} 
Note that this is a relaxation of Watkins' question that we negatively answered with \Tr{nonCay}. Our case analysis in the proof of Corollary~\ref{cor pm} is helpful for this problem: in cases 1 \& 3 it is not hard to obtain a positive answer. 

For graphs of degree higher than 3 we do not expect this to be true, as one can obtain such a graph as the cartesian product of an arbitrary finite vertex-transitive graph with the 2-way infinite path. But we expect that \fe\ $k\in \N$ \ti\ $f(k)\in \N$ \st\ every $k$-regular,  2-ended, vertex-transitive graph is $f(k)$-Cayley.

\begin{problem} \label{prob fk}
Provided lower and upper bounds for $f(k)$.
\end{problem} 

\medskip
Finally, it would be interesting to know whether \Cr{cor pm} generalises to graphs of arbitrary degree:

\begin{problem} \label{prob kcol}
Let $\g=(V,E)$ be a 2-ended, $k$-regular, vertex-transitive graph. Must \g be $k$-edge-colourable? Must $E$ decompose into $k$  perfect matchings?
\end{problem} 
Our harmonic function technique for \Cr{cor pm} seems useful for this general case.

\comment{
	\begin{lemma} \label{lem}

\end{lemma}
%%%%%%%%%
\begin{proof}

\end{proof}

\begin{problem} \label{}

\end{problem} 
}

\bibliographystyle{plain}
\bibliography{../collective}

%\extras{}

\end{document}

%% file: defs.tex
\usepackage[usenames]{color} 
\usepackage{amsthm,amssymb,amsmath,bbm,enumerate,graphicx,epsf,stmaryrd,accents}
\usepackage[bookmarks, colorlinks=false, breaklinks=true]{hyperref} %REMOVE FOR ARXIV 

\usepackage{authblk}

\newcommand{\comment}[1]{}
\newcommand{\COMMENT}[1]{}

\definecolor{darkgray}{rgb}{0.3,0.3,0.3}
\newcommand{\defi}[1]{{\color{darkgray}\emph{#1}}}

\comment{
	\begin{lemma}\label{}	
\end{lemma}
\begin{proof}

\end{proof}

\begin{theorem}\label{}
\end{theorem} 
\begin{proof} 	

\end{proof}

}

\newtheorem{proposition}{Proposition}[section]
\newtheorem{definition}[proposition]{Definition}
\newtheorem{theorem}[proposition]{Theorem}
\newtheorem{corollary}[proposition]{Corollary}

\newtheorem{lemma}[proposition]{Lemma}

\newtheorem{conjecture}{{Conjecture}}[section]

\newtheorem{problem}[conjecture]{{Problem}}

\newtheorem{examp}[proposition]{Example}%[section]

\newcommand{\FIG}{0}

\ifnum \NOTESON = 1 \newcommand{\note}[1]{ 

\hspace*{-30pt}
	{\color{blue}  NOTE: \color{Turquoise}{\small  \tt \begin{minipage}[c]{1.1\textwidth}  #1 \end{minipage} \ignorespacesafterend }} 
	
	}
\else \newcommand{\note}[1]{} \fi

\newcommand{\afsubm}[1]{ \ifnum \Debug = 1 {\mymargin{#1}}
\fi} %For notes on after-submission changes

\ifnum \Debug = 1 
\else  \fi

\ifnum \FIG = 1 
\else  \fi

\ifnum \FIG = 1 
\else  \fi

\ifnum \Debug = 1 \usepackage[notref,notcite]{showkeys}
\fi

\ifnum \COLORON = 0 \renewcommand{\color}[1]{}
\fi

\newcommand{\N}{\ensuremath{\mathbb N}}
\newcommand{\R}{\ensuremath{\mathbb R}}

\newcommand{\Z}{\ensuremath{\mathbb Z}}

\newcommand{\cf}{\ensuremath{\mathcal F}}

\newcommand{\eps}{\ensuremath{\epsilon}}

\newcommand{\sm}{\backslash}

\newcommand{\sydi}{\triangle}

\makeatletter
\DeclareRobustCommand{\cev}[1]{%
  \mathpalette\do@cev{#1}%
}
\newcommand{\do@cev}[2]{%
  \fix@cev{#1}{+}%
  \reflectbox{$\m@th#1\vec{\reflectbox{$\fix@cev{#1}{-}\m@th#1#2\fix@cev{#1}{+}$}}$}%
  \fix@cev{#1}{-}%
}
\newcommand{\fix@cev}[2]{%
  \ifx#1\displaystyle
    \mkern#23mu
  \else
    \ifx#1\textstyle
      \mkern#23mu
    \else
      \ifx#1\scriptstyle
        \mkern#22mu
      \else
        \mkern#22mu
      \fi
    \fi
  \fi
}

\makeatother
\newcommand{\nin}{\ensuremath{{n\in\N}}}

\newcommand{\seq}[1]{\ensuremath{(#1_n)_{n\in\N}}} 

\newcommand{\g}{\ensuremath{G\ }}
\newcommand{\G}{\ensuremath{G}}

\newcommand{\knl}{Kirchhoff's node law}
\newcommand{\kcl}{Kirchhoff's cycle law}
\newcommand{\are}{\vec{e}}
\newcommand{\arE}{\vec{E}}

\newcommand{\Cg}{Cayley graph}

\newcommand{\Lr}[1]{Lemma~\ref{#1}}

\newcommand{\Tr}[1]{Theorem~\ref{#1}}

\newcommand{\Sr}[1]{Section~\ref{#1}}

\newcommand{\Prr}[1]{Pro\-position~\ref{#1}}
\newcommand{\Prb}[1]{Problem~\ref{#1}}
\newcommand{\Cr}[1]{Corollary~\ref{#1}}

\newcommand{\lf}{locally finite}

\renewcommand{\iff}{if and only if}
\newcommand{\fe}{for every}

\newcommand{\st}{such that}

\newcommand{\ti}{there is}

\newcommand{\obda}{without loss of generality}

\newcommand{\labtequ}[2]{%\labtequc{#1}{#2}}
 \begin{equation} \label{#1} 	\begin{minipage}[c]{0.9\textwidth}  #2 \end{minipage} \ignorespacesafterend \end{equation} }

\newcommand{\mymargin}[1]{% <- dieses % verhindert ein ungewolltes Leerzeichen
 \ifnum \Debug = 1
  \marginpar{%
    \begin{minipage}{\marginparwidth}\small%
      \begin{flushleft}%
        {\color{blue}#1}%
      \end{flushleft}%
   \end{minipage}%
  }%
 \fi
}%

\newcommand{\extras}[1]{% <- dieses % verhindert ein ungewolltes Leerzeichen
 \ifnum \Debug = 1
\section{Extras} #1
 \fi
}%

\newcommand{\mySection}[2]{}

%% file: 2-ended.bbl
\begin{thebibliography}{10}

\bibitem{AlLiZh}
B.~Alspach, Y.-P. Liu, and C.-Q. Zhang.
\newblock Nowhere-{Zero} 4-{Flows} and {Cayley} {Graphs} on {Solvable}
  {Groups}.
\newblock {\em SIAM J.~Discrete Mathematics}, 9(1):151--154, February 1996.

\bibitem{BenIns}
I.~Benjamini.
\newblock Instability of the {Liouville} property for quasi-isometric graphs
  and manifolds of polynomial volume growth.
\newblock {\em {J.~Theor.~Probab.}}, 4(3):631--637, 1991.

\bibitem{BDKY}
I.~Benjamini, H.~Duminil-Copin, G.~Kozma, and A.~Yadin.
\newblock Disorder, {Entropy} and {Harmonic} {Functions}.
\newblock {\em The Annals of Probability}, 43(5):2332--2373, 2015.

\bibitem{BoCoDa}
A.~Bou-Rabee, W.~Cooperman, and P.~Dario.
\newblock Rigidity of harmonic functions on the supercritical percolation
  cluster.
\newblock arXiv:2303.04736.

\bibitem{BriQua}
S.~G. Brick.
\newblock Quasi-isometries and ends of groups.
\newblock {\em {J.~Pure Appl.~Algebra}}, 86(1):23--33, 1993.

\bibitem{CheYauDif}
S.~Y. Cheng and S.~T. Yau.
\newblock Differential equations on riemannian manifolds and their geometric
  applications.
\newblock {\em Communications on Pure and Applied Mathematics}, 28(3):333--354,
  1975.

\bibitem{ColMinHar}
T.~H. Colding and W.~P.~Minicozzi II.
\newblock Harmonic functions with polynomial growth.
\newblock {\em Journal of Differential Geometry}, 46(1):1--77, 1997.

\bibitem{csoka_invariant_2017}
E.~Csoka and G.~Lippner.
\newblock Invariant random perfect matchings in {Cayley} graphs.
\newblock {\em Groups, Geometry, and Dynamics}, 11(1):211--243, 2017.

\bibitem{DCSmiCon}
H.~Duminil-Copin and S.~Smirnov.
\newblock The connective constant of the honeycomb lattice equals
  $\sqrt{2+\sqrt{2}}$.
\newblock {\em Annals of mathematics}, 175(3):1653--1665, 2012.

\bibitem{planarPB}
A.~Georgakopoulos.
\newblock The boundary of a square tiling of a graph coincides with the poisson
  boundary.
\newblock {\em Invent.\ math.}, 203(3):773--821, 2016.

\bibitem{ma3h2}
A.~Georgakopoulos.
\newblock {\em Lecture Notes Markov Processes and Percolation Theory}.
\newblock 2021.
\newblock Available at {\small\tt
  https://warwick.ac.uk/fac/sci/maths/people/staff/agelos\_georgakopoulos/notesma3h2.pdf}.

\bibitem{squareRiemann}
A.~Georgakopoulos and C.~Panagiotis.
\newblock {Convergence of square tilings to the Riemann map}.
\newblock arXiv:1910.06886.

\bibitem{partite}
A.~Georgakopoulos and A.~Wendland.
\newblock {Presentations for Vertex Transitive Graphs}.
\newblock {\em J.\ Algebr.\ Comb.}, 55:795--826, 2022.

\bibitem{GriAna}
A.~Grigor'yan.
\newblock Analytic and geometric background of recurrence and non-explosion of
  the {Brownian} motion on {Riemannian} manifolds.
\newblock {\em Bulletin \ Am.\ Math.\ Soc.}, 36(2):135--249, 1999.

\bibitem{GriLiBou}
G.~Grimmett and Z.~Li.
\newblock {Bounds on connective constants of regular graphs}.
\newblock {\em Combinatorica}, 35(3):279--294, 2015.

\bibitem{GriLiLoc}
G.~R. Grimmett and Z.~Li.
\newblock {Locality of connective constants}.
\newblock {\em {Discrete Math.}}, 341(12):3483--3497, 2018.

\bibitem{GL16}
G.~R. Grimmett and Z.~Li.
\newblock Cubic graphs and the golden mean.
\newblock {\em Discrete Math.}, 343(1):111638, 2020.

\bibitem{GurNachRec}
O.~Gurel-Gurevich and A.~Nachmias.
\newblock Recurrence of planar graph limits.
\newblock {\em Annals of Mathematics}, 177(2):761--781, 2013.

\bibitem{HuaJosPol}
B.~Hua and J.~Jost.
\newblock Polynomial growth harmonic functions on groups of polynomial volume
  growth.
\newblock {\em Mathematische Zeitschrift}, 280(1):551--567, 2015.

\bibitem{HuJoLiGeo}
B.~Hua, J.~Jost, and S.~Liu.
\newblock Geometric analysis aspects of infinite semiplanar graphs with
  nonnegative curvature.
\newblock {\em J.\ Reine Angew.\ Math.}, 2015(700):1--36, March 2015.

\bibitem{KleNew}
B.~Kleiner.
\newblock A new proof of {Gromov}'s theorem on groups of polynomial growth.
\newblock {\em J.~Amer.~Math.~Soc.}, 23:815--829, 2010.

\bibitem{Leemann}
P.~H. Leemann.
\newblock {\em On subgroups and Schreier graphs of finitely generated groups.
  PhD Thesis}.
\newblock Universit\'{e} de Gen\'{e}ve, 2016.

\bibitem{LyonsBook}
Russell Lyons and Yuval Peres.
\newblock {\em Probability on Trees and Networks}.
\newblock Cambridge University Press, New York, 2016.
\newblock Available at {http://pages.iu.edu/~rdlyons/}.

\bibitem{LyoIns}
T.~Lyons.
\newblock Instability of the {Liouville} property for quasi-isometric
  {Riemannian} manifolds and reversible {Markov} chains.
\newblock {\em {J.~Differential Geom.}}, 26(1):33--66, 1987.

\bibitem{CThyperb}
S.~Markvorsen, S.~McGuinness, and C.~Thomassen.
\newblock Transient random walks on graphs and metric spaces with applications
  to hyperbolic surfaces.
\newblock {\em Proc.\ London Math.\ Soc.}, 64:1--20, 1992.

\bibitem{sab}
G.~Sabidussi.
\newblock {On a class of fixed-point-free graphs.}
\newblock {\em Proc.\ Am.\ Math.\ Soc.}, 9:800--804, 1958.

\bibitem{ShaTaoFin}
Y.~Shalom and T.~Tao.
\newblock A {Finitary} {Version} of {Gromov}'s {Polynomial} {Growth} {Theorem}.
\newblock {\em GAFA}, 20(6):1502--1547, 2010.

\bibitem{Wa90}
M.~E. Watkins.
\newblock Vertex-transitive graphs that are not {C}ayley graphs.
\newblock In {\em Cycles and rays ({M}ontreal, {PQ}, 1987)}, volume 301 of {\em
  NATO Adv. Sci. Inst. Ser. C Math. Phys. Sci.}, pages 243--256. Kluwer Acad.
  Publ., Dordrecht, 1990.

\bibitem{WatEdg}
M.~E. Watkins.
\newblock {Edge-Transitive Strips}.
\newblock {\em Disc.\ Math.}, 95(1):359--372, 1991.

\bibitem{woessBook}
Wolfgang Woess.
\newblock {\em Random walks on infinite graphs and groups}.
\newblock Cambridge University Press, 2002.

\bibitem{YauNon}
S.-T. Yau.
\newblock Nonlinear {Analysis} in {Geometry}.
\newblock {\em Enseign.~Math.}, 33(2):109--158, 1987.

\end{thebibliography}
